\documentclass[10pt, reqno]{amsart}

\usepackage{mathtools}
\mathtoolsset{showonlyrefs=true}

\usepackage{mathrsfs}
\usepackage{epsfig}
\usepackage{graphicx}
\usepackage{xcolor}
\usepackage{amsmath}
\usepackage{amsfonts}
\usepackage{amssymb}
\usepackage{amsthm}
\usepackage{amscd}
\usepackage{verbatim}

\usepackage{subcaption}
    \usepackage{tikz} 
    \usetikzlibrary{decorations.pathreplacing}
    \usetikzlibrary{arrows,shapes,backgrounds,3d}
\usetikzlibrary{decorations.pathreplacing,shapes,snakes,automata}

\usepackage[sc]{mathpazo}
\usepackage[T1]{fontenc} 
\usepackage{hyperref}
\usepackage{enumitem}
    \setenumerate{itemsep=5pt} % adjust space between list items
    \setitemize{itemsep=5pt} % adjust space between list items
\usepackage[font=footnotesize,labelfont=bf]{caption}
\usepackage{bm}
\usepackage{cite}

\theoremstyle{plain}

\newcounter{thing}
\newcounter{result}

\newtheorem{theorem}[result]{Theorem}
\newtheorem{proposition}[result]{Proposition}
\newtheorem{lemma}[result]{Lemma}

\newtheorem{problem}[thing]{Problem}
\theoremstyle{definition}
\newtheorem{definition}[result]{Definition}

\newtheorem{remark}[result]{Remark}
\newtheorem{example}[result]{Example}

\DeclareMathAlphabet{\mathpzc}{OT1}{pzc}{m}{it}

\def\mx{\langle x,x^*\rangle}

\renewcommand{\phi}{\varphi}
\newcommand{\C}{\mathbb{C}}
\newcommand{\E}{\mathbb{E}}
\newcommand{\R}{\mathbb{R}}
\newcommand{\N}{\mathbb{N}}
\newcommand{\Z}{\mathbb{Z}}

\newcommand{\G}{\mathrm{G}}
\newcommand{\D}{\mathrm{D}}
\renewcommand{\H}{\mathrm{H}}
\newcommand{\M}{\mathrm{M}}
\newcommand{\OO}{\mathrm{O}}
\renewcommand{\P}{\mathrm{P}}
\newcommand{\U}{\mathrm{U}}

\newcommand{\HH}{\mathfrak{H}}
\newcommand{\NN}{\mathfrak{N}}
\newcommand{\TT}{\mathrm{T}}

\newcommand{\V}{\mathcal{V}}
\newcommand{\W}{\mathcal{W}}

\newcommand{\Sym}{\operatorname{Sym}}

\newcommand{\diag}{\operatorname{diag}}

\newcommand{\tr}{\operatorname{tr}}
\newcommand{\ran}{\operatorname{ran}}

\renewcommand{\vec}[1]{{\bm{#1}}}

\newcommand{\inner}[1]{\langle #1 \rangle}
\newcommand{\norm}[1]{\| #1 \|}
\newcommand{\onorm}[1]{\| #1 \|_{\mathrm{op}}}
\newcommand{\cnorm}[1]{|\!|\!| #1 |\!|\!|}

\newcommand{\minimatrix}[4]{\begin{bmatrix} #1 & #2 \\ #3 & #4 \end{bmatrix}}
\allowdisplaybreaks
\begin{document}

\title[Norms on complex matrices]{Norms on complex matrices induced by complete homogeneous symmetric polynomials}

\author[K.~Aguilar]{Konrad Aguilar}
\address{Department of Mathematics, Pomona College, 610 N. College Ave., Claremont, CA 91711} 
\email{konrad.aguilar@pomona.edu}
\urladdr{\url{https://aguilar.sites.pomona.edu/}}

\author[\'A.~Ch\'avez]{\'Angel Ch\'avez}
\email{Angel.Chavez@pomona.edu}

\author[S.R.~Garcia]{Stephan Ramon Garcia}
\email{stephan.garcia@pomona.edu}
\urladdr{\url{http://pages.pomona.edu/~sg064747}}

\author[J.~Vol\v{c}i\v{c}]{Jurij Vol\v{c}i\v{c}}
\address{Department of Mathematical Sciences, University of Copenhagen,
Universitetspark 5, 
2100 Copenhagen \O, 
Denmark} 
\email{jv@math.ku.dk}

\thanks{Third named author supported by NSF grants DMS-1800123 and DMS-2054002.
Fourth named author supported by NSF grant DMS-1954709, and by Villum Fonden via the Villum Young Investigator grant (No. 37532)
}

\begin{abstract}
We introduce a remarkable new family of norms on the space of $n \times n$ complex matrices. These norms arise from the combinatorial properties of symmetric functions, and their construction and validation involve probability theory, partition combinatorics, and trace polynomials in noncommuting variables.  Our norms enjoy many desirable analytic and algebraic properties, such as an elegant determinantal interpretation and the ability to distinguish certain graphs that other matrix norms cannot.  Furthermore, they give rise to new dimension-independent tracial inequalities.  Their potential merits further investigation.
\end{abstract}

\maketitle

%%%%%%%%%%%%%%%%%%%%%%%%%%%%%%%%%%%%%%%%%%%%%%%%%%%%%%%%%%%%%%%%%%
%%%%%%%%%%%%%%%%%%%%%%%%%%%%%%%%%%%%%%%%%%%%%%%%%%%%%%%%%%%%%%%%%%
\section{Introduction}

In this note we introduce a family of norms on complex matrices.
These are initially defined in terms of certain symmetric functions of eigenvalues of complex Hermitian matrices.
The fact that we deal with eigenvalues, as opposed to their absolute values, is notable.
First, it prevents standard machinery, such as the theory of symmetric gauge functions, from applying.
Second, the techniques used to establish that we indeed have norms are more complicated than
one might expect.
For example, combinatorics, probability theory, and
Lewis' framework for group invariance in convex matrix analysis each play key roles.

These norms on the Hermitian matrices are of independent interest.  
They can be computed recursively or directly read from
the characteristic polynomial.
Moreover, our norms distinguish certain pairs of graphs which the standard norms
(operator, Frobenius, Schatten-von Neumann, Ky Fan) cannot distinguish.

Our norms extend in a natural and nontrivial manner to all complex matrices.
These extensions of our original norms involve partition combinatorics and trace polynomials in noncommuting variables.
A Schur convexity argument permits our norms to be bounded below in terms of the mean eigenvalue of a matrix.

These norms, their unusual construction, and their potential applications suggest a host of open problems.
We pose several at the end of the paper.

%%%%%%%%%%%%%%%%%%%%%%%%%%%%%%%%%
\subsection{Notation}
Denote by $\N$, $\R$, and $\C$, respectively,
the set of natural numbers, real numbers, and complex numbers.
Let $\H_n(\C)$ denote the set of $n \times n$ complex Hermitian matrices and $\M_n(\C)$ the set of $n \times n$
complex matrices.
Denote the eigenvalues of $A \in \H_n(\C)$ by
$\lambda_1(A) \geq \lambda_2(A)  \geq \cdots \geq \lambda_n(A)$
and define 
\begin{equation*}
\vec{\lambda}(A) = \big( \lambda_1(A), \lambda_2(A),\ldots, \lambda_n(A) \big) \in \R^n.
\end{equation*}
We may use $\vec{\lambda}$ and $\lambda_1,\lambda_2,\ldots,\lambda_n$ if the matrix $A$ is clear from context.
Let $\diag(x_1,x_2,\ldots,x_n) \in \M_n(\C)$ denote the $n \times n$ diagonal matrix
with diagonal entries $x_1,x_2,\ldots,x_n$, in that order.  If $\vec{x} = (x_1,x_2,\ldots,x_n)$ is understood from context,
we may write $\diag(\vec{x})$ for brevity.

%%%%%%%%%%%%%%%%%%%%%%%%%%%%%%%%%
\subsection{Complete homogeneous symmetric polynomials}
The \emph{complete homogeneous symmetric (CHS) polynomial} of degree $d$ in the
$n$ variables $x_1, x_2, \ldots, x_n$ is
\begin{equation}\label{eq:CHSp}
h_d(x_1,x_2,\ldots,x_n) \quad=\!\! 
\sum_{1 \leq i_1 \leq \cdots \leq i_{d} \leq n} x_{i_1} x_{i_2}\cdots x_{i_d},
\end{equation}
the sum of all degree $d$ monomials in $x_1,x_2,\ldots,x_n$ \cite[Sec.~7.5]{StanleyBook2}.  For example,
\begin{align*}
h_0(x_1,x_2)&= 1,\\
h_1(x_1,x_2)&= x_1+x_2,\\
h_2(x_1,x_2)&= x_1^2+x_1 x_2+x_2^2, \quad\text{and}\\
h_3(x_1,x_2)&= x_1^3+ x_1^2 x_2+x_1 x_2^2 +x_2^3.
\end{align*}
Elementary combinatorics confirms that there are precisely $\binom{n+d-1}{d}$
summands in the definition \eqref{eq:CHSp}.  We often write $h_d(\vec{x})$, in which $\vec{x} = (x_1,x_2,\ldots,x_n) \in \R^n$, when the number of variables
is clear from context.  

For $d$ even and $\vec{x} \in \R^n$, Hunter proved that $h_{d}(\vec{x}) \geq 0$,
with equality if and only if $\vec{x} = \vec{0}$ \cite{Hunter}.
This  is not obvious because some of the summands that comprise $h_d(\vec{x})$
(for $d$ even) may be negative.
Hunter's theorem has been reproved many times; 
see 
\cite[Lem.~3.1]{Barvinok},
\cite{Baston},
\cite[p.~69 \& Thm.~3]{BGON},
\cite[Cor.~17]{GOOY},
\cite[Thm.~2.3]{Roventa},
and \cite[Thm.~1]{Tao}.

%%%%%%%%%%%%%%%%%%%%%%%%%%%%%%%%%
\subsection{Partitions and traces} 
A \emph{partition} of $d\in \N$ is an $r$-tuple $\vec{\pi}=(\pi_1, \pi_2, \ldots, \pi_r) \in \N^r$ such that
$\pi_1 \geq \pi_2 \geq \cdots \geq \pi_r$ and
$\pi_1+ \pi_2 + \cdots + \pi_r = d$; the number of terms $r$ depends on the partition $\vec{\pi}$.
We write $\vec{\pi} \vdash d$ if $\vec{\pi}$ is a partition of $d$.

For $\vec{\pi} \vdash d$, define the symmetric polynomial
\begin{equation*}
p_{\vec{\pi}}(x_1, x_2, \ldots, x_n)=p_{\pi_1}p_{\pi_1}\cdots p_{\pi_r},
\end{equation*}
in which $p_k(x_1,x_2, \ldots, x_n)=x_1^k+x_2^k+\cdots +x_n^k$ are the power sum symmetric polynomials.
If the length of $\vec{x} = (x_1,x_2,\ldots,x_n)$ is clear from context, we often write
$p_{\vec{\pi}}(\vec{x})$ and $p_k(\vec{x})$, respectively.
Another expression for \eqref{eq:CHSp} is
\begin{equation}\label{eq:hdgen}
 h_{d}(x_1,x_2,\ldots, x_n)=\sum_{\vec{\pi} \,\vdash\, d} \frac{p_{\vec{\pi}}(x_1,x_2,\ldots, x_n)}{z_{\vec{\pi}}} ,
\end{equation}   
in which the sum runs over all partitions $\vec{\pi}=(\pi_1, \pi_2, \ldots,\pi_r)$ of $d$ and 
\begin{equation}\label{eq:zI}
z_{\vec{\pi}} = \prod_{i \geq 1} i^{m_i} m_i!,
\end{equation} 
where $m_i$ is the multiplicity of $i$ in $\vec{\pi}$ \cite[Prop.~7.7.6]{StanleyBook2}.
For example, if $\vec{\pi} = (4,4,2,1,1,1)$, then $z_{\vec{\pi}}= (1^3 3!) (2^1 1!) (4^2 2!) = 384$ \cite[(7.17)]{StanleyBook2}.
The integer $z_{\vec{\pi}}$ is precisely the Hall inner product of $p_{\vec{\pi}}$ with itself, in symmetric function theory.

If $A \in \H_n(\C)$ has eigenvalues $\vec{\lambda} = (\lambda_1,\lambda_2,\ldots,\lambda_n)$, then
\begin{equation}\label{eq:pTrace}
p_{\vec{\pi}}(\vec{\lambda}) =p_{\pi_1}(\vec{\lambda})p_{\pi_2}(\vec{\lambda})\cdots p_{\pi_r}(\vec{\lambda})
=(\tr A^{\pi_1})(\tr A^{\pi_2})\cdots (\tr A^{\pi_{r}}).
\end{equation} 
This connects eigenvalues, traces, and partitions to symmetric polynomials.

%%%%%%%%%%%%%%%%%%%%%%%%%%%%%%%%%%%%%%%%%%
%%%%%%%%%%%%%%%%%%%%%%%%%%%%%%%%%%%%%%%%%%
%%%%%%%%%%%%%%%%%%%%%%%%%%%%%%%%%%%%%%%%%%
\subsection{Main results}
The following theorem provides a family of novel norms on the space $\H_n(\C)$ of $n \times n$ Hermitian matrices.
Some special properties of these norms are discussed in Section \ref{Section:Properties}.

\begin{theorem}\label{Theorem:Main}
For even $d \geq 2$, the following is a norm on $\H_n(\C)$:
\begin{equation*}
\cnorm{A}_{d}=\big( h_{d}\big(\lambda_1(A), \lambda_2(A), \ldots, \lambda_n(A)\big)\big)^{1/d}.
\end{equation*}
\end{theorem}

For example, equations \eqref{eq:hdgen} and \eqref{eq:pTrace} yield trace-polynomial representations
{\small
\begin{align}
\cnorm{A}_2^2 &= \frac{1}{2}\big(\tr(A^2) + (\tr A)^2 \big), \label{eq:TwoInner} \\[5pt]
\cnorm{A}_4^4 &= \frac{1}{24}\big( (\tr A)^4 + 6 (\tr A)^2 \tr(A^2) + 3 (\tr(A^2))^2 + 8 (\tr A) \tr(A^3) +  6 \tr(A^4) \big).\label{eq:FourExpect}
\end{align}
}
Theorem \ref{Theorem:Main} is nontrivial for several reasons.
\begin{enumerate}[leftmargin=*]
\item The sums \eqref{eq:CHSp} and \eqref{eq:hdgen} that characterize $h_{d}(\vec{\lambda}(A))$ may contain negative summands.
For example, $(\tr A) \tr (A^3)$ in \eqref{eq:FourExpect} can be negative for Hermitian $A$:
consider $A = \diag(-2,1,1,1)$.

\item The sums that define these norms do \underline{not} involve the absolute values of the eigenvalues of $A$.
Theorem \ref{Theorem:Main} does not follow from standard considerations, 
but rather from delicate properties of multivariate symmetric polynomials.

\item The relationship between the spectra of (Hermitian) $A$, $B$, and $A+B$,
conjectured by A.~Horn in 1962 \cite{HornAlfred}, was only established in 1998-9
by Klyachko \cite{Klyachko} and Knutson--Tao \cite{Knutson}.  Therefore, the triangle inequality is difficult to establish.
Even if $A$ and $B$ are diagonal, the result is not obvious; see \eqref{eq:Triangle}.  In fact, even in the ``easy'' case of positive
diagonal matrices this result has been rediscovered and republished many times; see Remark \ref{Remark:TriangleIneq}.

\item Passing from the diagonal case to the general Hermitian case is not straightforward.
We emphasize again that standard techniques like symmetric gauge functions are not applicable because of (a).
Our proof of this step involves Lewis' framework for group invariance in convex matrix analysis \cite{LewisGroup}.

\item 
A remarkable general approach to norms on $\R^n$ arising from multivariate homogeneous polynomials 
is due to Ahmadi, de Klerk, and Hall \cite[Thm.~2.1]{AKH}. 
Unfortunately, this does not apply in our setting because the convexity of the even-degree CHS polynomials is hard to establish directly.  
In fact, Theorem \ref{Theorem:Main} together with \cite[Thm.~2.1]{AKH} imply convexity.
\end{enumerate}

\begin{example}
Because CHS norms do not rely upon the absolute values of the eigenvalues
of a Hermitian matrix (that is, its singular values), they can sometimes distinguish singularly (adjacency) cospectral graphs
(graphs with the same singular values) that are not adjacency cospectral.  This feature is not enjoyed by many standard norms
(e.g., operator, Frobenius, Schatten--von Neumann, Ky Fan). For example,
\begin{equation*}
K = 
\begin{bmatrix}
0 & 1 & 1\\
1 & 0 & 1\\
1 & 1 & 0
\end{bmatrix},
\end{equation*}
which has eigenvalues $2,-1,-1$, is the adjacency matrix for the complete
graph on three vertices.
The graphs with adjacency matrices
\begin{equation*}
A = \minimatrix{K}{0}{0}{K}
\qquad \text{and} \qquad
B = \minimatrix{0}{K}{K}{0}
\end{equation*}
are singularly cospectral but not cospectral: their eigenvalues are $-1, -1, -1, -1, \break 2, 2$ and 
$-2,-1,-1,1,1,2$, respectively.  
Moreover, $\cnorm{A}_6^6=120 \neq 112 = \cnorm{B}_6^6$.
\end{example}

The norms of Theorem \ref{Theorem:Main} extend in a natural and nontrivial 
fashion to the space $\M_n(\C)$ of all $n \times n$ complex matrices.  

\begin{theorem}\label{Theorem:Complex}
Let $d \geq 2$ be even and let $\vec{\pi}=(\pi_1, \pi_2, \ldots,\pi_r)$ be a partition of $d$. 
Define $\TT_{\vec{\vec{\pi}}} : \M_{n}(\C)\to \R$
by setting $\TT_{\vec{\pi}}(A)$ to be $1/{d\choose d/2}$ times the sum over the $\binom{d}{d/2}$ 
possible locations to place $d/2$ adjoints ${}^*$ among the $d$ copies of $A$ in
\begin{equation*}
(\tr \underbrace{AA\cdots A}_{\pi_1})
(\tr \underbrace{AA\cdots A}_{\pi_2})
\cdots
(\tr \underbrace{AA\cdots A}_{\pi_r}).
\end{equation*}
Then
\begin{equation}\label{eq:Extended}
\cnorm{A}_{d}= \bigg( \sum_{\vec{\pi} \,\vdash\, d} \frac{\TT_{\vec{\pi}}(A)}{z_{\vec{\pi}}}\bigg)^{1/d},
\end{equation} 
in which the sum runs over all partitions $\vec{\pi}$ of $d$ and $z_{\vec{\pi}}$ is defined in \eqref{eq:zI},
is a norm on $\M_n(\C)$ that restricts to the
norm on $\H_n(\C)$ given by Theorem \ref{Theorem:Main}.
\end{theorem}

If $A = A^*$, observe that \eqref{eq:Extended} coincides with the norm of Theorem \ref{Theorem:Main}
in light of \eqref{eq:hdgen} and \eqref{eq:pTrace}.
We prove Theorems \ref{Theorem:Main} and \ref{Theorem:Complex} in the next two sections of this paper.

\begin{example}
The two partitions of $d=2$ satisfy $z_{(2)} = 2$ and $z_{(1,1)} = 2$.  There are $\binom{2}{1} = 2$ ways to place two adjoints ${}^*$
in a string of two $A$s.  Therefore,
\begin{align*}
\TT_{(2)}(A) &= \frac{1}{2} \big(\tr(A^*A)+\tr(AA^*) \big)  =  \tr(A^*A) , \quad \text{and} \\
\TT_{(1,1)}(A) &= \frac{1}{2}\big( (\tr A^*)(\tr A) +(\tr A)(\tr A^*) \big) =(\tr A^*)(\tr A) ,
\end{align*}
so
\begin{equation}\label{eq:cn2}
\cnorm{A}_2^2 = \frac{1}{2} \tr(A^*A) + \frac{1}{2}(\tr A^*)(\tr A).
\end{equation}
If $A = A^*$, this simplifies to the norm \eqref{eq:TwoInner} on $\H_n(\C)$, as expected.
\end{example}

\begin{example}
The five partitions of $d=4$ satisfy
$z_{(4)} = 4$,
$z_{(3,1)} = 3$,
$z_{(2,2)} = 8$,
$z_{(2,1,1)} = 4$,
and $z_{(1,1,1,1)} = 24$.
There are $\binom{4}{2} = 6$ ways to place two adjoints ${}^*$
in a string of four $A$s.  For example, 
\begin{align*}
6\TT_{(3,1)}(A)
&= (\tr A^*A^*A)(\tr A) + (\tr A^*AA^*)(\tr A) + (\tr A^*AA)(\tr A^*) \\
&\qquad + (\tr AA^*A^*)(\tr A) +(\tr AA^*A)(\tr A^*) +(\tr AAA^*)(\tr A^*)\\
&=3 \tr (A^{*2}A)(\tr A) +3 (\tr A^2 A^*)(\tr A^*).
\end{align*}
Some computation and \eqref{eq:Extended} reveal that
\begin{align}
\cnorm{A}_4^4 &= \frac{1}{24} \big(
(\tr A)^2 \tr(A^*)^2 + \tr(A^*)^2 \tr(A^2) + 4 \tr(A) \tr(A^*) \tr(A^*A)    \nonumber \\ 
& \qquad + 
 2 \tr(A^*A)^2 + (\tr A)^2 \tr(A^{*2}) + \tr(A^2) \tr(A^{*2}) +  4 \tr(A^*) \tr(A^*A^2) \nonumber \\
  & \qquad + 4 \tr(A) \tr(A^{*2}A) + 2 \tr(A^*AA^*A) + 4 \tr(A^{*2}A^2) \big). \label{eq:cn4}
\end{align}
If $A = A^*$, this simplifies to the norm \eqref{eq:FourExpect} on $\H_n(\C)$, as expected.
\end{example}

Because of their origins in terms of complete homogeneous symmetric polynomials,
we sometimes refer to the norm $\cnorm{\cdot}_d$ as the \emph{CHS norm} of order $d$.  
The notation $\norm{\cdot}$ is used, occasionally with subscripts, for other norms.

In the Hermitian case, the norm $\cnorm{\cdot}_d$ can be directly extracted from the Taylor expansion of an explicit rational function
(Theorem \ref{Theorem:Characteristic}).  The general situation is elegantly summarized in a determinantal formula.

\begin{theorem}\label{Theorem:Determinant}
Let $A\in\M_n(\C)$. For $d$ even, $\binom{d}{d/2} \cnorm{A}_d^d$ is the coefficient of $z^{d/2}\overline{z}^{d/2}$ in the Taylor expansion of $\det(I-zA-\overline{z}A^*)^{-1}$ about the origin.
\end{theorem}
Helton and Vinnikov showed that polynomials of the form $p=\det(I-zA-\overline{z}A^*)\in\C[z,\overline{z}]$
are precisely the real-zero polynomials in $\C[z,\overline{z}]$ \cite{HV}. That is, they are characterized by the conditions
$p(0)=1$ and that $x\mapsto p(\alpha x)$ has only real zeros for every $\alpha\in\C$. Properties of such polynomials are studied within the framework of hyperbolic \cite{Ren} and stable \cite{Wag} polynomials.

This paper is structured as follows.
Sections \ref{Section:ProofHermitian} and \ref{Section:ProofComplex} contain the proofs of 
Theorems \ref{Theorem:Main} and \ref{Theorem:Complex}, respectively.
Section \ref{Section:Properties} surveys the remarkable properties of the CHS norms, including Theorem 
\ref{Theorem:Determinant}.
We pose several open questions in Section \ref{Section:Open}.

\noindent\textbf{Acknowledgments.} We thank the anonymous referee for many (forty three to be exact!)
helpful comments and suggestions.

%%%%%%%%%%%%%%%%%%%%%%%%%%%%%%%%%%%%%%%%%%%%%%%%%%%%%%%%%%%%%%%%%%%%
%%%%%%%%%%%%%%%%%%%%%%%%%%%%%%%%%%%%%%%%%%%%%%%%%%%%%%%%%%%%%%%%%%%%
%%%%%%%%%%%%%%%%%%%%%%%%%%%%%%%%%%%%%%%%%%%%%%%%%%%%%%%%%%%%%%%%%%%%
\section{Proof of Theorem \ref{Theorem:Main}}\label{Section:ProofHermitian}

Let $d\geq 2$ be even.
We prove that $\HH:\H_n(\C)\to\R$ defined by
\begin{equation}\label{eq:NN}
\HH(A) = h_{d}\big( \lambda_1(A), \lambda_2(A),\ldots, \lambda_n(A)   \big)^{1/d}
\end{equation}
is a norm.
Hunter's theorem ensures that $\HH(A) \geq 0$ and, moreover, that $\HH(A)= 0$ if and only if $A = 0$
(the nonnegativity of $\HH$ already follows from \eqref{eq:NewHunter} below).  Since $\HH(cA) = |c| \HH(A)$
for all $A \in \H_n(\C)$ and $c \in \R$, it suffices to prove that $\HH$ satisfies the triangle inequality.
This is accomplished by combining Lewis' framework for group invariance in convex matrix analysis \cite{LewisGroup}
with a probabilistic approach to the complete homogeneous symmetric polynomials \cite{Barvinok, Roventa, Tao}, as we now explain.

%%%%%%%%%%%%%%%%%%%%%%%%%%%%%%%%%%%%%%%%%%%%%%%%%%%%%%%%%%%%%%%%%%%%
\subsection{Group invariance}
Let $\V$ be a finite-dimensional $\R$-inner product space.
The adjoint $\phi^*$ of a linear map $\phi: \V\to \V$ satisfies
$\inner{\phi^*(X), Y }= \inner{ X, \phi(Y) }$ for all $X,Y \in \V$.
We say that $\phi$ is \emph{orthogonal} if $\phi^*\circ \phi$ is the identity map on $\V$. 
Let $\OO(\V)$ denote the set of all orthogonal linear maps on $V$.  For a subgroup $\G \subseteq \OO(\V)$,
we say that $f: \V\to \R$ is \emph{$\G$-invariant} if 
$f( \phi(X))=f(X)$ for all $\phi\in G$ and $X\in \V$.

\begin{definition}[Def.~2.1 of \cite{LewisGroup}]\label{Definition:NDS}
$\delta: \V\to \V$ is a \emph{$\G$-invariant normal form} if 
\begin{enumerate}
\item $\delta$ is $\G$-invariant,
\item For each $X\in \V$, there is an $\phi\in \OO(\V)$  such that $X=\phi\big( \delta(X)\big)$, and
\item $\inner{X, Y} \leq \inner{ \delta(X), \delta(Y) }$ for all $X,Y\in \V$.
\end{enumerate}
In this case, $( \V, \G, \delta)$ is called a \emph{normal decomposition system}.
\end{definition} 

Suppose that $(\V, \G, \delta)$ is a normal decomposition system
and $\W \subseteq \V$ is a subspace.  The \emph{stabilizer} of $\W$ in $\G$ is 
$\G_{\W} = \{ \phi \in \G : \phi(\W)=\W\}$.  For convenience, we restrict the domain of each $\phi \in \G_{\W}$
and consider $\G_{\W}$ as a subset of $\OO(\W)$.  

Our interest in this material stems from the next result.

\begin{lemma}[Thm.~4.3 of \cite{LewisGroup}]\label{Lemma:Lewis}
Let $(\V, \G, \delta)$ and $(\W, \G_{\W}, \delta|_{\W})$ be normal decomposition systems with 
$\ran \delta \subseteq \W \subseteq \V$. Then a $\G$-invariant function $f:\V\to \R$ is convex if and only if its restriction to $\W$ is convex. 
\end{lemma}

Let $\V=\H_n(\C)$ denote the $\R$-inner product space of complex Hermitian matrices,
endowed with the inner product $\inner{ X, Y} =\tr (X Y)$, and let $\U_n(\C)$ denote the group of $n \times n$ unitary matrices;
see Remark \ref{Remark:Trace} for more details about this inner product.
For each $U \in \U_n(\C)$, define a linear map $\phi_U: \V\to \V$ by 
$\phi_U(X)=UXU^*$.
Observe that $\phi_U\circ \phi_V=\phi_{UV}$ and hence
\begin{equation*}
\G=\{\phi_U : U\in \U_n(\C)\}
\end{equation*}
is a group under composition.  
Since $\phi_U^*=\phi_{U^*}$, we conclude that $\G$ is a subgroup of $\OO(\V)$. 
Moreover, the function \eqref{eq:NN} is $\G$-invariant.

Let $\W=\D_n(\R)$ denote the set of real diagonal matrices. 
Then $\W$ inherits an inner product from $\V$ and 
$\G_{\W} = \{ \phi_P : P \in \P_n\}$,
in which $\P_n$ denotes the set of $n \times n$ permutation matrices.
Define $\delta: \V\to \V$ by 
\begin{equation*}
\delta(X)=\diag \big(\lambda_1(X), \lambda_2(X), \ldots, \lambda_n(X)\big),
\end{equation*} 
the $n \times n$ diagonal matrix with $\lambda_1(X) \geq \lambda_2(X) \geq \cdots \geq \lambda_n(X)$ on its diagonal.
Observe that $\ran \delta \subseteq \W$ since the eigenvalues of a Hermitian matrix are real.
We maintain all of this notation below.

\begin{lemma}\label{Lemma:Normal}
$(\V, \G,\delta)$ and $(\W, \G_{\W}, \delta|_{\W})$ are normal decomposition systems.
\end{lemma}

\begin{proof}
We first show that $(\V, \G,\delta)$ is a normal decomposition system.
(a) Since eigenvalues are invariant under similarity, $\delta$ is $\G$-invariant.
(b) For $X \in \V$, the spectral theorem provides a $U \in \U_n(\C)$ such that
$X = U\delta(X)U^* = \phi_U( \delta(A) )$.
(c) For $X,Y \in \V$, note that $\tr XY \leq \tr \delta(X)\delta(Y)$ \cite[Thm.~2.2]{LewisHermitian}; see Remark \ref{Remark:Trace}.

We now show that $(\W,\G_{\W}, \delta|_{\W})$ is a normal decomposition system.
(a) $\delta|_{\W}$ is $\G_{\W}$-invariant since $\delta(\phi_P(X))=\delta(PXP^*) = \delta(X)$ for every $X \in \W$ and $P \in \P_n$.
(b) Let $X \in \W$.  Since $X$ is diagonal there exists a $P \in \P_n$ such that $X = P\delta(X)P^* = \phi_P(\delta(X))$.  
(c) The diagonal elements of a diagonal matrix are its eigenvalues.  Consequently, this property is inherited from $\V$; see Remark \ref{Remark:DiagTrace}.
\end{proof}

%%%%%%%%%%%%%%%%%%%%%%%%%%%%%%%%%%%%%%%%%%%%%%%%%%%%%%%%%%%%%%%%%%
\subsection{CHS polynomials as expectations}
Let $\vec{\xi}=(\xi_1, \xi_2,\ldots, \xi_n)$ be a vector of independent standard exponential random variables
\cite[(20.10)]{Billingsley}, and let
$\vec{x} = (x_1,x_2,\ldots,x_n) \in \R^n$.
Since $\E[\xi_i^k]=k!$ for $i=1,2,\ldots,n$ \cite[Ex.~21.3]{Billingsley}, we deduce that
\begin{align*}
\E[ \inner{ \vec{\xi}, \vec{x} }^d]
&= \E[(\xi_1 x_1 + \xi_2 x_2 \cdots +\xi_n x_n)^d] \\
&= \E\left[ \sum_{k_1+k_2+\cdots+ k_n=d}  \frac{d!}{k_1!\,k_2!\,\cdots\, k_n!} \xi_1^{k_1}\xi_2^{k_2}\cdots \xi_n^{k_n}x_1^{k_1}x_2^{k_2}\cdots x_n^{k_n}  \right]\\
&= \sum_{k_1+k_2+\cdots+ k_n=d}  \frac{d!}{k_1!\,k_2!\,\cdots\, k_n!} \E\left[ \xi_1^{k_1}\xi_2^{k_2}\cdots \xi_n^{k_n}x_1^{k_1}x_2^{k_2}\cdots x_n^{k_n}  \right]\\
&= d! \sum_{k_1+k_2+\cdots+ k_n=d}  \frac{\E[ \xi_1^{k_1}] \E[\xi_2^{k_2}]\cdots \E[\xi_n^{k_n}]}{k_1!\,k_2!\,\cdots\, k_n!}  x_1^{k_1}x_2^{k_2}\cdots x_n^{k_n}  \\
&= d!\sum_{k_1+k_2+\cdots+ k_n=d} x_1^{k_1}x_2^{k_2}\cdots x_n^{k_n} \\
&= d! \, h_d(\vec{x})
\end{align*}
for integral $d \geq 1$ by the linearity of expectation and the independence of the $\xi_1,\xi_2,\ldots,\xi_n$;
see Remark \ref{Remark:Probability}.
Now suppose that $d$ is even.  Then
\begin{equation}\label{eq:NewHunter}
 h_d(\vec{x}) = \frac{1}{d!} \E\big[ \,| \inner{ \vec{\xi} , \vec{x} } |^{d}\,  \big] \geq 0.
\end{equation}
For $\vec{x},\vec{y}\in \R^n$, Minkowski's inequality implies that 
\begin{align*}
\Big(\E\big[ \,| \inner{ \vec{\xi} , \vec{x}+\vec{y}}|^{d}\,  \big]\Big)^{1/d}
\leq \Big(\E\big[ \,| \inner{ \vec{\xi} , \vec{x} }|^{d}\,  \big]\Big)^{1/d}+\Big(\E\big[ \,| \inner{ \vec{\xi} , \vec{y}}|^{d}\,  \big]\Big)^{1/d},
\end{align*} 
and hence (for $d$ even)
\begin{equation}\label{eq:Triangle}
\big[h_{d}(\vec{x}+\vec{y})\big]^{1/d}\leq \big[h_{d}(\vec{x})\big]^{1/d}+\big[h_{d}(\vec{y})\big]^{1/d}.
\end{equation}

%%%%%%%%%%%%%%%%%%%%%%%%%%%%%%%%%%%%%%%%%%%%%%%%%%%%%%%%%%%%%%%%%%
\subsection{Conclusion}
Recall the definition \eqref{eq:NN} of the function $\HH:\H_n(\C)\to \R$.
The inequality \eqref{eq:Triangle} ensures that the restriction of $\HH$ to $\D_n(\R)$
satisfies the triangle inequality.  For $A,B \in \D_n(\R)$ and $t \in [0,1]$, note that
\begin{equation*}
\HH(tA + (1-t)B) \leq \HH(tA) + \HH((1-t)B) = t \HH(A) + (1-t) \HH(B)
\end{equation*}
by \eqref{eq:Triangle} and homogeneity.  Thus, $\HH$ is a convex function on $\D_n(\R)$.
Since $\HH$ is $\G$-invariant, we conclude from Lemma \ref{Lemma:Lewis} that $\HH$
is convex on $\H_n(\C)$.  It satisfies the triangle inequality on $\H_n(\C)$ since it is convex and homogeneous:
\begin{equation*}
\tfrac{1}{2} \HH(A+B) = \HH(\tfrac{1}{2}A + \tfrac{1}{2}B) \leq \tfrac{1}{2}\HH(A) + \tfrac{1}{2} \HH(B).
\end{equation*}
Consequently, $\HH(\,\cdot\,)$ is a norm on $\H_n(\C)$. \qed

%%%%%%%%%%%%%%%%%%%%%%%%%%%%%%%%%%%%%%%%%%%%%%%%%%%%%%%%%%%%%%%%%%
\subsection{Remarks}
We collect here a few remarks about the proof of Theorem \ref{Theorem:Main}.

\begin{remark}\label{Remark:Trace}
Consider the inner product $\inner{X,Y} = \tr  (XY)$ on $\H_n(\C)$;
it is the restriction of the Frobenius inner product to $\H_n(\C)$.
The inequality 
\begin{equation}\label{eq:vonNeumann}
\tr (XY) \leq \tr \delta(X)\delta(Y) \qquad \text{for $X,Y \in \H_n(\C)$}
\end{equation}
is due to von Neumann \cite{MR0157874} and has been reproved
many times; see 
de S\'a \cite{MR1275626},
Lewis \cite[Thm.~2.2]{LewisHermitian}, 
Marcus \cite{Marcus},
Marshall \cite{MR2759813}, 
Mirsky \cite[Thm.~1]{Mirsky},
Richter \cite[Satz.~1]{Richter},
Rendl and Wolkowicz \cite[Cor.~3.1]{Rendl},
and
Theobald \cite{Theobald}.
\end{remark}

\begin{remark}\label{Remark:DiagTrace}
For diagonal matrices, the inequality \eqref{eq:vonNeumann} 
is equivalent to a classical rearrangement result:
$\inner{ \vec{x} ,\vec{y} }\leq\inner{ \widetilde{\vec{x}},\widetilde{\vec{y}} }$, in which
where $\widetilde{\vec{x}} \in \R^n$ has the components of $\vec{x}=(x_1,x_2, \ldots, x_n)$ in decreasing order \cite[Thm.~368]{Hardy}.
\end{remark}

\begin{remark}\label{Remark:Probability}
For even $d$, \eqref{eq:NewHunter} implies the nonnegativity of the
CHS polynomials.  This probabilistic approach appears in the comments on the blog entry \cite{Tao},  
and in \cite[Lem.~12]{Sra}, which cites \cite{Barvinok}.
There are many other proofs of the nonnegativity of the even-degree CHS
polynomials.  Of course, there is Hunter's inductive proof \cite{Hunter}.
Roven\c{t}a and Temereanc\u{a} used divided differences \cite[Thm.~3.5]{Roventa}.
Recently, B\"ottcher, Garcia, Omar and O'Neill \cite{BGON} employed a spline-based approach
suggested by Olshansky
after Garcia, Omar, O'Neill, and Yih obtained it as a byproduct of investigations into
numerical semigroups \cite[Cor.~17]{GOOY}.
\end{remark}

\begin{remark}\label{Remark:Referee}
The CHS polynomials are a special case of the more general \emph{Schur polynomials}
\begin{equation*}
s_{(n_1,n_2,\ldots,n_N)}(u_1,u_2,\ldots,u_N) = \frac{ \det( u_j^{n_i+N-i})}{ \det(u_j^{N-i})},\qquad
n_1 \geq n_2 \geq \cdots \geq n_N \geq 0.
\end{equation*}
These polynomials are also monomial-positive, homogeneous, and symmetric in the $u_j$, and moreover, carry representation-theoretic content. A natural question is whether the family of CHS norms on Hermitian matrices is part of a larger family of ``Schur norms.'' In other words, is there a converse of Hunter's positivity result is valid for other Schur polynomials? 
Khare and Tao proved that this is not the case \cite[Prop.~6.3]{KhareTao}.  We thank the referee for pointing out this
direction of inquiry.
\end{remark}

\begin{remark}\label{Remark:TriangleIneq}
We stress that the inequality \eqref{eq:Triangle} permits $\vec{x}, \vec{y}\in \R^n$; that is, with no positivity assumptions.   
For $p \in \N$, the similar inequality 
\begin{equation}\label{eq:WeakTriangle}
h_p(\vec{x}+\vec{y})^{1/p} \leq h_p(\vec{x})^{1/p}+h_p(\vec{y})^{1/p}
\qquad \text{for $\vec{x},\vec{y} \in \R_{\geq 0}^n$}
\end{equation}
has been rediscovered several times.
According to McLeod \cite[p.~211] {McLeod} and Whiteley \cite[p.~49]{Whiteley},
it was first conjectured by A.C.~Aitken.
Priority must be given to Whiteley \cite[eq.~(5)]{Whiteley}, whose paper appeared in 1958.
McLeod's paper was received on March 16, 1959, although he was unaware of Whiteley's proof:
``To the best of my knowledge, no proof of [\eqref{eq:WeakTriangle}] exists so far in the literature.''
For more exotic inequalities along the lines of \eqref{eq:WeakTriangle}, see \cite{Sra}.
\end{remark}

%%%%%%%%%%%%%%%%%%%%%%%%%%%%%%%%%%%%%%%%%%%%%%%%%%%%%%%%%%%%%%%%%%%%
\section{Proof of Theorem \ref{Theorem:Complex}}\label{Section:ProofComplex}
The first step in the proof of Theorem \ref{Theorem:Complex} is a general complexification result.
Let $\V$ be a complex vector space with a conjugate-linear involution $v\mapsto v^*$.
Suppose that the real subspace $\V_\R=\{v\in \V:v=v^*\}$ of $*$-fixed points is endowed with a norm $\| \cdot\|$.
For each $v \in \V$ and $t \in \R$, we have $e^{it}v+e^{-it}v^* \in \V_{\R}$.
Note that the path $t\mapsto \|e^{it}v+e^{-it}v^*\|$ is continuous for each $v\in \V$.

\begin{proposition}\label{p:gennorm}
For even $d \geq 2$, the following 
is a norm on $\V$ that extends $\|\cdot\|$:
\begin{equation}\label{eq:gennorm}
\NN_d(v)= \bigg( \frac{1}{2\pi \binom{d}{d/2}}\int_0^{2\pi}\|e^{it}v+e^{-it}v^*\|^d\,dt \bigg)^{1/d}.
\end{equation}
\end{proposition}

\begin{proof}
If $v\in \V_\R$, then $\|e^{it}v+e^{-it}v^*\|=|2\cos t|\|v\|$. Moreover, $\NN_d(v)=\|v\|$ since
\begin{equation*}
\int_0^{2\pi} |2\cos t|^d\,dt = 2\pi \binom{d}{d/2}.
\end{equation*}
Next we verify that $\NN_d$ is a norm on $\V$.

\medskip\noindent\textsc{Positive definiteness}. 
The nonnegativity of $\|\cdot\|$ on $\V_{\R}$ and \eqref{eq:gennorm} ensure that
$\NN_d$ is nonnegative on $\V$.
If $v\in \V\backslash\{0\}$, then $v = u + iu'$, where $u = \frac{1}{2}(v + v^*)$
and $u' = \frac{1}{2}(-iv + iv^*)$ belong to $\V_{\R}$.  Now $u,u'$ cannot both be zero, so the map
$t\mapsto \|e^{it}v+e^{-it}v^*\| = \| 2 \cos(t)u + 2 \sin(t) u'\|$ is continuous and positive almost everywhere.
Thus, $\NN_d(v)\neq0$.

\medskip\noindent\textsc{Absolute homogeneity}. 
For $r>0$ and $\theta\in\R$, we have $\NN_d(( re^{i\theta})v ) = r \NN_d(e^{i\theta}v)  =r\NN_d(v)$
by the $\R$-homogeneity of $\|\cdot\|$ and the periodicity of the integrand in \eqref{eq:gennorm}.

\medskip\noindent\textsc{Triangle inequality}. For $u,v\in \V$,
\begin{align*}
&\bigg(\int_0^{2\pi}\|e^{it}(u+v)+e^{-it}(u+v)^*\|^d\,dt\bigg)^{1/d} \\
&\qquad \leq 
\bigg(\int_0^{2\pi}\big(\|e^{it}u+e^{-it}u^*\|+\|e^{it}v+e^{-it}v^*\|\big)^d\,dt\bigg)^{1/d} \\
&\qquad \leq \bigg(\int_0^{2\pi}\|e^{it}u+e^{-it}u^*\|^d\,dt\bigg)^{1/d}
+\bigg(\int_0^{2\pi}\|e^{it}v+e^{-it}v^*\|^d\,dt\bigg)^{1/d}\,,
\end{align*}	
where the first inequality holds by monotonicity of power functions and the triangle inequality for $\|\cdot\|$, 
and the second inequality holds by the triangle inequality for the $L^d$ norm on the space $C[0,2\pi]$.
\end{proof}

There are several natural complexifications of a real Banach space \cite{Munoz}. 
The extensions $\NN_d$ in \eqref{eq:gennorm} are special since they preserve some of the analytic and algebraic properties of the original norm. 
Namely, we will show that when the extension $\NN_d$ is applied to the norm $\cnorm{\cdot}_d$ on $\H_n(\C)$, one obtains a norm on 
$\M_n(\C)$ whose power is a trace polynomial; this does not happen, for example, if one uses the minimal or the projective complexification of a norm (in which case the resulting norm is not an algebraic function).

%%%%%%%%%%%%%%%%%%%
Let $\mx$ be the free monoid generated by $x$ and $x^*$. 
Let $|w|$ denote the length of a word $w\in\mx$ 
and let $|w|_x$ count the occurrences of $x$ in $w$.  
For $A\in \M_n(\C)$, let $w(A)\in \M_n(\C)$ be the natural evaluation of $w$ at $A$.
For example, if $w = xx^*x^2$, then $|w| = 4$, $|w|_x = 3$, and $w(A) = A A^* A^2$.

\begin{lemma}\label{l:expr}
Let $d\geq 2$ be even and let $\vec{\pi}=(\pi_1,\dots,\pi_r)$ be a partition of $d$. If $A\in \M_n(\C)$, then
\begin{equation}\label{eq:int2poly}
\begin{split}
&\frac{1}{2\pi} \int_0^{2\pi}\tr(e^{it}A+e^{-it}A^*)^{\pi_1}\cdots \tr(e^{it}A+e^{-it}A^*)^{\pi_r}\,dt \\
&\qquad = \sum_{\substack{
		w_1,\dots,w_r \in \mx \colon \\
		|w_j|=\pi_j\ \forall j \\
		|w_1\cdots w_r|_x = \frac{d}{2}
}} \tr w_1(A)\cdots\tr w_r(A)\,.
\end{split}
\end{equation}
\end{lemma}

\begin{proof}
For every Laurent polynomial $f\in \C[z,z^{-1}]$ with the constant term $f_0$ we have $\int_0^{2\pi}f(e^{it})\,dt=2\pi f_0$. Let us view
\begin{equation*}
f=\tr(zA+z^{-1}A^*)^{\pi_1}\cdots \tr(zA+z^{-1}A^*)^{\pi_r}
\end{equation*}
as a Laurent polynomial in $z$.  Its constant term is
\begin{equation*}
f_0=\sum_{w_1,\dots,w_r} \tr w_1(A)\cdots\tr w_r(A)
\end{equation*}
where the sum runs over all words $w_1,w_2,\ldots,w_r$ in $\mx$ with $|w_j|=\pi_j$ such that the number of occurrences of $x$ in $w_1 w_2\cdots w_r$ equals the number of occurrences of $x^*$ in $w_1 w_2\cdots w_r$. Thus, \eqref{eq:int2poly} follows.
\end{proof}

Given a partition $\vec{\pi}=(\pi_1,\dots,\pi_r)$ of $d$ and $A\in \M_n(\C)$ let 
\begin{equation*}
\TT_{\vec{\pi}}(A)=\frac{1}{\binom{d}{d/2}}\sum_{\substack{
	w_1,\dots,w_r \in \mx \colon \\
	|w_j|=\pi_j\ \forall j \\
	|w_1\cdots w_r|_x = \frac{d}{2}
}} \tr w_1(A)\cdots\tr w_r(A).
\end{equation*}

We now complete the proof of Theorem \ref{Theorem:Complex}.
The conjugate transpose $A \mapsto A^*$ is a real structure on $\M_n(\C)$.
The corresponding real subspace of $*$-fixed points is $\H_n(\C)$. We apply Proposition \ref{p:gennorm} to the norm $\cnorm{\cdot}_d$ on $\H_n(\C)$ and 
obtain its extension $\NN_d(\cdot)$ to $\M_n(\C)$ defined by \eqref{eq:gennorm}.
The fact that $\NN_d(A)$ admits a trace-polynomial expression as in \eqref{eq:Extended}  
follows from \eqref{eq:hdgen} and Lemma \ref{l:expr}. 

Concretely, if $A \in \M_n(\C)$ and $\NN_d(B) = \|B\|$ is the CHS-norm over Hermitian matrices $B$,
then by Proposition \ref{p:gennorm}, the following is a norm on $\M_n(\C)$:
\begin{align*}
\NN_d(A)
&\overset{\eqref{eq:gennorm}}{=} \left( \frac{1}{2\pi \binom{d}{d/2}} \int_0^{2 \pi} h_d( \vec{\lambda}(e^{it} A + e^{-it}A^*))^d\,dt\right)^{1/d} \\
&\overset{\eqref{eq:hdgen}}{=} \left( \frac{1}{2\pi \binom{d}{d/2}} \int_0^{2 \pi} \sum_{\vec{\pi} \,\vdash\, d} \frac{p_{\vec{\pi}}( \vec{\lambda}(e^{it} A  + e^{-it}A^*))}{z_{\pi}}\,dt \right)^{1/d} \\
&\overset{\eqref{eq:pTrace}}{=} \left( \frac{1}{\binom{d}{d/2}} \sum_{\vec{\pi} \,\vdash\, d} \frac{1}{z_{\vec{\pi}}\cdot 2\pi} \int_0^{2\pi} \tr(e^{it} A + e^{-it}A^*)^{\pi_1} \cdots
\tr (e^{it} A + e^{-it} A^*)^{\pi_r} \,dt \right)^{1/d} \\
&\overset{\eqref{eq:int2poly}}{=} \left( \frac{1}{\binom{d}{d/2}} \sum_{\vec{\pi} \,\vdash\, d} \frac{\TT_{\vec{\pi}}(A) \binom{d}{d/2} }{z_{\vec{\pi}}} \right)^{1/d},
\end{align*}
which concludes the proof.
\qed

\begin{remark}
Proving that \eqref{eq:Extended} is a norm relies crucially on Theorem \ref{Theorem:Main}, which states that its restriction to $\H_n(\C)$ is a norm. 
On the other hand, demonstrating that \eqref{eq:Extended} is a norm in a direct manner seems arduous. 
To a certain degree, this mirrors the current absence of general certificates for dimension-independent positivity of trace polynomials in $x,x^*$ (see \cite{KSV} for the analysis in a dimension-fixed setting).
\end{remark}

\begin{remark}
For any $A \in \M_n(\C)$ and $t \in [0,2\pi]$, the matrices $e^{it}A+e^{-it}A^*$
are Hermitian.  Thus, $\cnorm{e^{it}A+e^{-it}A^*}_d$ can be defined as in Theorem \ref{Theorem:Main} and hence
\begin{equation}\label{eq:IntegralForm}
\cnorm{A}_d =  \bigg( \frac{1}{2\pi \binom{d}{d/2}}\int_0^{2\pi} \cnorm{e^{it}A+e^{-it}A^*}_d^d\,dt \bigg)^{1/d}.
\end{equation}
\end{remark}

\begin{remark}\label{Remark:Czz}
Here is another way to restrict $\cnorm{\cdot}_d$ to the Hermitian matrices. 
The proof of Lemma \ref{l:expr} shows that $\binom{d}{d/2} \cnorm{A}_d^d$ is the coefficient of $z^{d/2}\bar{z}^{d/2}$ in 
\begin{equation*}
\cnorm{zA+\overline{z}A^*}_d^d \in \C[z,\overline{z}].
\end{equation*}
\end{remark}

%%%%%%%%%%%%%%%%%%%%%%%%%%%%%%%%%%%%%%%%%%%%%%%%%%%%%%%%%%%%%%%%%%%%
\section{Properties of CHS norms}\label{Section:Properties}
We now establish several properties of the CHS norms.
First, we show how the CHS norm of a Hermitian matrix can be computed rapidly and exactly
from its characteristic polynomial 
and recursion (Subsection \ref{Subsection:Characteristic}).
This leads quickly to the determinantal interpretation presented in the introduction (Subsection \ref{Subsection:Determinant}). 
Next, we identify those CHS norms induced by inner products (Subsection \ref{Subsection:Inner}).
In Subsection \ref{Subsection:Mean}, we use Schur convexity to provide a lower bound on the CHS norms
in terms of the trace seminorm on $\M_n(\C)$.
We discuss monotonicity properties in Subsection \ref{Subsection:Monotone}
and symmetric tensor powers in Subsection \ref{Subsection:Tensor}.

%%%%%%%%%%%%%%%%%%%%%%%%%%%%%%%%%%%%%%%%%%%%%%%%%%%%%%%%%%%%

\subsection{Exact computation via characteristic polynomial}\label{Subsection:Characteristic}
The CHS norm of a Hermitian matrix can be exactly computed from its characteristic polynomial.
The following theorem involves only formal series manipulations.

\begin{theorem}\label{Theorem:Characteristic}
Let $p_A(x)$ denote the characteristic polynomial of $A \in \H_n(\C)$.
For $d\geq 2$ even, $\cnorm{A}_d^d$ is the $d$th coefficient in the Taylor expansion of 
\begin{equation*}
\frac{1}{\det(I - x A)} = \frac{1}{ x^n p_A(1/x)}
\end{equation*}
about the origin. 
\end{theorem}

\begin{proof}
Let $p_A(x) = (x-\lambda_1)(x-\lambda_2)\cdots ( x - \lambda_n)$.
For $|x|$ small, \cite[(1)]{Tao} provides
\begin{equation*}
\sum_{d=0}^{\infty} h_d(\lambda_1,\lambda_2,\ldots,\lambda_n) x^d 
= \prod_{k=1}^n \frac{1}{1-\lambda_k x} 
= \frac{1}{x^n}\prod_{k=1}^n \frac{1}{x^{-1}-\lambda_k } 
= \frac{1}{x^n p_A(1/x)} ;
\end{equation*}
the apparent singularity at the origin is removable.   Now observe that
\begin{equation*}
\prod_{k=1}^n \frac{1}{1-\lambda_k x}  = \frac{1}{\det \diag(1 - \lambda_1 x,1 - \lambda_2x,  \ldots, 1 - \lambda_n x)} = \frac{1}{\det(I - xA)}
\end{equation*}
by the spectral theorem.
\end{proof}

\begin{example}\label{Example:Fibo}
Let $A =\small \minimatrix{1}{1}{1}{0}$.  Then $p_A(z) = x^2 - x - 1$ and
\begin{equation*}
\frac{1}{x^2 p_A(1/x)} = \frac{1}{1 -x-x^2} = \sum_{n=0}^{\infty} f_{n+1} x^n,
\end{equation*}
in which $f_n$ is the $n$th Fibonacci number; these are defined by 
$f_{n+2} = f_{n+1} + f_n$ and $f_0 = 0$ and $f_1 = 1$.  Thus,
$\cnorm{A}_d^d = f_{d}$ for even $d \geq 2$.  
\end{example}

\begin{remark}
If $A \in \H_n(\C)$ is fixed,
the sequence $h_d(\lambda_1,\lambda_2,\ldots,\lambda_n)$ satisfies a constant-coefficient recurrence of order $n$ since
its generating function is a rational function whose denominator has degree $n$.  Solving such a recurrence is elementary,
so one can compute $\norm{A}_d$ for $d=2,4,6,\ldots$ via this method.  
\end{remark}

\begin{remark}
For small $d$, there is a simpler method.
Since $p_A(x)$ is monic, it follows that $\widetilde{p_A}(x) = x^n p_A(1/x)$ has constant term $1$.  For small $x$, we have
\begin{equation*}
\sum_{d=0}^{\infty} h_d(\lambda_1,\lambda_2,\ldots,\lambda_n) x^d 
= \frac{1}{\widetilde{p_A}(x) } = \frac{1}{1 - (1-\widetilde{p_A}(x))} = \sum_{d=0}^{\infty} (1-\widetilde{p_A}(x))^d
\end{equation*}
so the desired $h_d(\lambda_1,\lambda_2,\ldots,\lambda_n)$ can be computed by the expanding the geometric series
to the appropriate degree.
\end{remark}

\begin{remark}
For $d \geq 1$, the Newton--Gerard identities imply
\begin{equation*}
h_d(x_1,x_2,\ldots,x_n) =  \frac{1}{d} \sum_{i=1}^d h_{d-i}(x_1,x_2,\ldots,x_n) p_i(x_1,x_2,\ldots,x_n);
\end{equation*}
see \cite[\S10.12]{Seroul}. 
For $A \in \H_n(\C)$ and $d \geq 2$ even, it follows that
\begin{equation}\label{eq:Recursive}
h_d(\vec{\lambda}(A)) = \frac{1}{d} \sum_{i=1}^d h_{d-i}(\vec{\lambda}(A))  \tr (A^i),
\end{equation}
which can be used to compute $\cnorm{A}_d^d = h_d( \vec{\lambda}(A))$ recursively.
\end{remark}

\begin{remark}
If $H,K \in \H_n(\C)$, then $\det(I-xH) = \det(I-xK)$ if and only if they are unitarily similar.
However, $H=\diag(1,0)$ and $K=\diag(1,-1)$ give
\begin{equation*}
\frac{1}{\det(I - xH)} = \frac{1}{1-x} = \sum_{j=0}^{\infty} z^j
\qquad \text{and} \qquad
\frac{1}{\det(I - xK)} = \frac{1}{1-x^2}  = \sum_{k=0}^{\infty} z^{2k},
\end{equation*}
so $\cnorm{H}_d =\cnorm{K}_d$ for even $d \geq 2$.  Of course,
the odd-indexed coefficients (the complete homogeneous symmetric polynomials of odd degree) do not agree.
\end{remark}

%%%%%%%%%%%%%%%%%%%%%%%%%%%%%%%%%%%%%%%%%%%%%%
\subsection{Determinantal interpretation}\label{Subsection:Determinant}
The material of the previous subsection leads to the determinantal interpretation 
(Theorem \ref{Theorem:Determinant}) stated in the introduction.
We restate (and prove) the result here for convenience:

\begin{theorem}
Let $A\in\M_n(\C)$. For $d$ even, $\binom{d}{d/2} \cnorm{A}_d^d$ is the coefficient of $z^{d/2}\overline{z}^{d/2}$ in the Taylor expansion of $\det(I-zA-\overline{z}A^*)^{-1}$ about the origin.
\end{theorem}

\begin{proof}
If $H \in \H_n(\C)$, the coefficient of $x^d$ in $\det(I-xH)^{-1}$ is $\cnorm{H}_d^d$ by Theorem \ref{Theorem:Characteristic}. 
By plugging in $H=zA+\overline{z}A^*$ and treating the resulting expression as a series in $z$ and $\overline{z}$, Remark \ref{Remark:Czz} 
implies that the coefficient of $z^{d/2}\overline{z}^{d/2}$ equals $\binom{d}{d/2} \cnorm{A}_d^d$.
\end{proof}

\begin{example}
Let $A = \minimatrix{0}{1}{0}{0}$.  Then
\begin{equation*}
\det(I-zA-\overline{z}A^*)^{-1} = \frac{1}{1-\overline{z} z} = \sum_{n=0}^{\infty} \overline{z}^n z^n,
\end{equation*}
and hence $\norm{A}_d^d = \binom{d}{d/2}^{-1}$ for even $d \geq 2$.
\end{example}

\begin{example}
For
\begin{equation*}
A = 
\begin{bmatrix}
0 & 1 & 0 \\
0 & 0 & 1 \\
1 & 0 & 0
\end{bmatrix},
\quad \text{we have} \quad
\det(I-zA-\overline{z}A^*)^{-1} = 
\frac{1}{1 -z^3 -3 z \overline{z}- \overline{z}^3}.
\end{equation*}
Computer algebra reveals that $\cnorm{A}_2^2 = \cnorm{A}_4^4 = \frac{3}{2}$, 
$\cnorm{A}_6^6 = \frac{29}{20}$, and
$\cnorm{A}_8^8 = \frac{99}{70}$.
\end{example}

\begin{example}
The matrices
\begin{equation*}
A=
\begin{bmatrix}
0 & 0 & 0 \\
0 & 1 & i \\
0 & i & -1
\end{bmatrix}
\qquad \text{and} \qquad
B=
\begin{bmatrix}
0 & 0 & 1\\
0 & 0 & i \\
1 & i & 0
\end{bmatrix}
\end{equation*}
satisfy
\begin{equation*}
\det(I-zA-\overline{z}A^*)^{-1} = \frac{1}{1 - 4 z \overline{z}} =  \det(I-zB-\overline{z}B^*)^{-1},
\end{equation*}
so $\cnorm{A}_d = \cnorm{B}_d$ for even $d\geq 2$.
These matrices are not similar (let alone unitarily similar)
since $A$ is nilpotent of order two and $B$ is nilpotent of order three.
\end{example}

\begin{remark}
In terms of the Laplace operator $\Delta=\frac{\partial^2}{\partial z \, \partial \overline{z}}$,
Theorem \ref{Theorem:Determinant} states that for even $d$,
\begin{equation*}
d!\cnorm{A}_d^d = \Delta^{d/2} \frac{1}{\det(I-zA-\overline{z}A^* ) }(0)\,.
\end{equation*}
\end{remark} 

%%%%%%%%%%%%%%%%%%%%%%%%%%%%%%%%%%%%%%%%%%%%%%%%%%%%%%%%%%%%

%%%%%%%%%%%%%%%%%%%%%%%%%%%%%%%%%%%%%%%%%%%%%%%%%%%%%%%%%%%%
\subsection{Inner products}\label{Subsection:Inner}
Theorem \ref{Theorem:Complex} says that $\cnorm{\cdot}_d$ is a norm on $\M_n(\C)$ for even $d\geq 2$.  
It is natural to ask when these norms are induced by an inner product.

\begin{theorem}
The norm $\cnorm{\cdot}_d$ on $\M_n(\C)$ (and its restriction to $\H_n(\C)$) is induced by an inner product if and only if
$d=2$ or $n=1$
\end{theorem}

\begin{proof}
If $n=1$ and $d \geq 2$ is even, then $\cnorm{A}_d$ is a fixed positive multiple of $ |a|$ for all $A = [a] \in \M_1(\C)$.
Thus, $\cnorm{\cdot}_d$ on $\M_1(\C)$ is induced by a positive multiple of the inner product
$\inner{A,B} = \overline{b}a$, in which $A=[a]$ and $B = [b]$.

If $d=2$ and $n \geq 1$, then $\cnorm{A}_2^2 = \tfrac{1}{2}\tr(A^*A) + \tfrac{1}{2}(\tr A) \tr(A^*)$,
which is induced by the inner product
$\inner{A,B} = \tfrac{1}{2}\tr(B^*A) + \tfrac{1}{2}(\tr B^*)(\tr A)$ on $\M_n(\C)$.  

It suffices to show that in all other cases
the norm $\cnorm{A}_{d}=( h_{d}( \vec{\lambda}(A)) )^{1/d}$ on $\H_n(\R)$ does not arise from an inner product.
For $n \geq 2$, let $A = \diag(1,0,0,\ldots)$ and $B= \diag(0,1,0,\ldots,0) \in \H_n(\R)$.  Then
$\cnorm{A}_d^2 = \cnorm{B}_d^2 = 1$.
Next observe that
$\cnorm{A+B}_d^2 = (d+1)^{2/d}$
since there are exactly $d+1$ nonzero summands, each equal to $1$, in the evaluation of $h_{d}(\vec{\lambda}(A+B))$.
Because of cancellation, a similar argument shows that
$\cnorm{A-B}_d^2 = 1$.
A result of Jordan and von Neumann says that a vector space norm $\norm{\cdot}$ arises from an inner product
if and only if it satisfies the parallelogram identity
$\norm{ \vec{x} + \vec{y}}^2 + \norm{ \vec{x} - \vec{y} }^2
= 2 \big( \norm{ \vec{x} }^2 + \norm{ \vec{y} }^2 )$
for all $\vec{x}, \vec{y}$ \cite{Jordan}.
If $\cnorm{\cdot}_d$ satisfies the parallelogram identity, then
$(d+1)^{2/d} + 1 = 2(1+1)$;
that is, $(d+1)^2 = 3^d$.  The solutions are $d=0$ (which does not yield an inner product) and $d=2$ (which, as we showed above, does).
Thus, for $n \geq 2$ and $d\geq 2$,
the norm $\cnorm{\cdot}_d$ on $\H_n(\C)$ does not arise from an inner product.
\end{proof}

%%%%%%%%%%%%%%%%%%%%%%%%%%%%%%%%%%%%%%%%%%%%%%%%%%%%%%%%%%%%%%%%%%%%
\subsection{A tracial lower bound}\label{Subsection:Mean}
Each CHS norm on $\M_n(\C)$ is bounded below by an explicit positive multiple of the trace seminorm.
That is, the CHS norms of a matrix can be related to its mean eigenvalue.

\begin{theorem}\label{Theorem:Mean}
For $A \in \M_n(\C)$ and $d\geq 2$ even,
\begin{equation*}
\cnorm{A}_d \geq  \binom{n+d-1}{d}^{1/d} \frac{| \tr A|}{n}
\end{equation*}
with equality if and only if $A$ is a multiple of the identity.
\end{theorem}

\begin{proof}
Let $d \geq 2$ be even.
For $\vec{x} = (x_1,x_2,\ldots,x_n)\in \R^n$, let $\widetilde{\vec{x}} = (\widetilde{x_1},\widetilde{x_2},\ldots,\widetilde{x_n})$
denote its decreasing rearrangement (the notation $\vec{x}^{\downarrow}$ is frequently used in the literature).
Then $\vec{x}$ \emph{majorizes} $\vec{y}$, denoted $\vec{x} \succeq \vec{y}$, if
\begin{equation*}
\sum_{i=1}^k \widetilde{x_i} \geq \sum_{i=1}^k \widetilde{y_i}
\quad \text{for $k=1,2,\ldots,n$, and} \quad \sum_{i=1}^n x_i = \sum_{i=1}^n y_i.
\end{equation*}
The even-degree complete homogeneous symmetric polynomials are Schur convex \cite[Thm.~1]{Tao}.
That is, $h_{d}(\vec{x}) \geq h_{d}(\vec{y})$ whenever $\vec{x} \succeq \vec{y}$,  with equality if and only if 
$\vec{x}$ is a permutation of $\vec{y}$.  

Let $A\in \M_n(\C)$ and define $B(t)=e^{it} A+e^{-it} A^*$ for $t \in \R$. 
Then $\vec{\lambda}(B(t))$ majorizes
$\vec{\mu}(t) = (\mu(t), \mu(t),\ldots, \mu(t))\in \R^n$, in which 
$\mu(t)=\tr B(t)/n$.  Thus,
\begin{equation*}
\cnorm{B(t)}_d^d=
h_d\big( \vec{\lambda}(B(t))\big)\geq h_d\big(\vec{\mu}(t)\big)=\mu(t)^d \binom{n+d-1}{d}
\end{equation*} 
with equality if and only if $B(t) = \mu(t) I$.
It follows from \eqref{eq:IntegralForm} that
\begin{equation}\label{eq:Admu}
\cnorm{A}_d
\geq\Bigg(\frac{ \binom{n+d-1}{d}}{2\pi {\binom{d}{d/2}}}\int_0^{2\pi} \mu(t)^d\,dt\Bigg)^{1/d}.
\end{equation}
Combine this with
\begin{align*}
\int_0^{2\pi} \mu(t)^d\,dt
&= \int_0^{2\pi} \left(\frac{ \tr B(t)}{n}  \right)^d\,dt 
= \frac{1}{n^d} \int_0^{2\pi} \big( e^{it} \tr A + e^{-it} \tr (A^*) \big)^d\,dt \\
&=  \frac{1}{n^d}\sum_{k=0}^d \binom{d}{k}(\tr A^*)^{d-k}(\tr A)^k  \int_0^{2\pi} e^{i(2k-d)t} \,dt \\
&=\frac{2\pi}{n^d} {\binom{d}{d/2}} |\tr A|^d
\end{align*}
and get the desired inequality.
The continuity of the integrand ensures that equality occurs in \eqref{eq:Admu} if and only if 
$B(t) = \mu(t) I$ for all $t \in \R$.  An operator-valued Fourier expansion reveals that
$e^{it} A+e^{-it} A^* = ( \sum_{n \in \Z} \widehat{\mu}(n) e^{int} ) I$,
so $A = \hat{\mu}(1) I$. Conversely, equality holds in \eqref{eq:Admu} 
if $A$ is a multiple of the identity.
\end{proof}

\begin{remark}
For each fixed $n \geq 1$, the constant $\binom{n+d-1}{d}^{1/d}$ in Theorem \ref{Theorem:Mean}
tends to $1$ from above as $d \to \infty$.  Therefore, $\cnorm{A}_d \geq \frac{1}{n} |\tr A|$ for all $A \in \M_n(\C)$.
\end{remark}

%%%%%%%%%%%%%%%%%%%%%%%%%%%%%%%%%%%%%%%%%%%%%%%%%%%%%%%%%
\subsection{Monotonicity}\label{Subsection:Monotone}
The next result shows how CHS norms relate to each other.
For Hermitian matrices, the first inequality below is superior to the second.

\begin{theorem}
Let $2 \leq p < q$ be even.
\begin{enumerate}
\item If $A \in \H_n(\C)$, then $(p!)^{1/p}\cnorm{A}_p \leq (q!)^{1/q}\cnorm{A}_q$.
\item If $A \in \M_n(\C)$, then $\big(\binom{p}{p/2}p!\big)^{1/p} \cnorm{A}_p \leq \big(\binom{q}{q/2}q! \big)^{1/q} \cnorm{A}_q$.
\end{enumerate}
\end{theorem}

\begin{proof}
\noindent(a)
Let $A \in \H_n(\C)$ have eigenvalues $\vec{\lambda} = (\lambda_1,\lambda_2,\ldots,\lambda_n)$, listed in decreasing order,
and let $\vec{\xi} = (\xi_1,\xi_2,\ldots,\xi_n)$ be a random vector, in which $\xi_1,\xi_2,\ldots,\xi_n$
are independent standard exponential random variables.  Let $d\geq 2$ be even and consider the random variable $X = \inner{ \vec{\xi}, \vec{\lambda}}$.
Then \eqref{eq:NewHunter} ensures that
\begin{equation}\label{eq:Probable}
(d!)^{1/d} \cnorm{A}_d = \big(  d!\, h_{d}(\vec{\lambda}) \big)^{1/d} 
=\E\big[ \,| \inner{ \vec{\xi} , \vec{\lambda} }|^{d}\,  \big]^{1/d} = \E[ |X|^{d} ]^{1/d} = \norm{X}_{L^{d}}.
\end{equation}
Since we are in a probability space (in particular, a finite measure space), 
$\norm{X}_{L^p} \leq \norm{X}_{L^q}$ for $1 \leq p < q < \infty$.
For $2 \leq p< q$ even, this yields the desired inequality.

\medskip\noindent(b)
Let $A \in \M_n(\C)$ and let $2 \leq p < q$ be even.  
For $t \in [0,2\pi]$, (a) ensures that
\begin{equation*}
\cnorm{e^{it} A + e^{-it} A^*}_p^p \leq \frac{(q!)^{p/q}}{p!}\cnorm{e^{it} A + e^{-it} A^*}_q^p.
\end{equation*}
Consider $f(t) = \cnorm{e^{it} A + e^{-it} A^*}_q$ as an element of $L^p[0,2\pi]$.
H\"older's inequality and \eqref{eq:IntegralForm} imply the desired inequality:
\begin{align*}
\cnorm{A}_p 
&= \bigg(\frac{1}{2\pi {\binom{p}{p/2}}}\int_0^{2\pi} \cnorm{ e^{it} A+e^{-it} A^*}_p^p\,dt\bigg)^{1/p} \\
&\leq \bigg(\frac{1}{2\pi {\binom{p}{p/2}}} \bigg)^{1/p} \bigg( \frac{(q!)^{p/q}}{p!}\int_0^{2\pi} \cnorm{ e^{it} A+e^{-it} A^*}_q^p\,dt\Bigg)^{1/p} \\
&\leq \frac{(q!)^{1/q}}{(p!)^{1/p}} \bigg(\frac{1}{2\pi {\binom{p}{p/2}}} \bigg)^{1/p} \norm{f}_{L^p} 
\leq \frac{(q!)^{1/q}}{(p!)^{1/p}} \bigg(\frac{1}{2\pi {\binom{p}{p/2}}} \bigg)^{1/p} (2\pi)^{\frac{1}{p}-\frac{1}{q}}\norm{f}_{L^q}\\
&\leq \frac{(q!)^{1/q}}{(p!)^{1/p}} \bigg(\frac{1}{2\pi {\binom{p}{p/2}}} \bigg)^{1/p} (2\pi)^{\frac{1}{p}-\frac{1}{q}}\bigg(\int_0^{2\pi} \cnorm{ e^{it} A+e^{-it} A^*}_q^q\,dt\Bigg)^{1/q} \\
&= \frac{(q!)^{1/q}}{(p!)^{1/p}} \bigg(\frac{1}{\binom{p}{p/2}} \bigg)^{1/p} (2\pi)^{-\frac{1}{q}} \bigg(2\pi {\binom{q}{q/2}} \bigg)^{1/q} \cnorm{A}_q^q\\
&\leq \frac{\big(\binom{q}{q/2}q! \big)^{1/q}}{ \big(\binom{p}{p/2}p!\big)^{1/p} } \cnorm{A}_q^q. \qedhere
\end{align*}
\end{proof}

\begin{remark}
The previous result suggests that suitable constant multiples of the CHS norms may be preferable in
some circumstances.  However, the benefits appear to be outweighed by the cumbersome
nature of these constants.
\end{remark}

\begin{remark}\label{Remark:Submult}
For $A,B\in\M_n(\C)$,
\begin{align*}
2\cnorm{AB}_2^2 
&= \tr(AB)\tr((AB)^*)+\tr((AB)^*AB) \\
&\le 2\tr(A^*A)\tr(B^*B) \\
&\le 2\big(\tr(A)\tr(A^*)+\tr(A^*A)\big)\big(\tr(B)\tr(B^*)+\tr(B^*B)\big) \\
& = 8\cnorm{A}_2^2\cnorm{B}_2^2,
\end{align*}
so $2\cnorm{\cdot}_2$ is submultiplicative. Actually, $2$ is the smallest constant independent of $n$ with this property,
since $$J=\minimatrix{0}{1}{0}{0}$$ satisfies
$\cnorm{JJ^*}_2 = 1= 2\cnorm{J}_2\cnorm{J^*}_2$.
\end{remark}

%%%%%%%%%%%%%%%%%%%%%%%%%%%%%%%%%%%%%%%%%%%%
\subsection{Symmetric Tensor Powers}\label{Subsection:Tensor}
Let $\V$ denote an $n$-dimensional $\R$-inner product space with orthonormal basis $\vec{v}_1, \vec{v}_2,\ldots,\vec{v}_n$.  The $k$th \emph{tensor power} of $\V$ is
the $n^k$-dimensional $\R$-inner product space
$\V^{\otimes k}$ spanned by the \emph{simple tensors}
\begin{equation}\label{eq:BasisTensors}
\vec{v}_{i_1} \otimes \vec{v}_{i_2} \otimes \cdots \otimes \vec{v}_{i_k},
\end{equation}
with these simple tensors forming an orthonormal basis of $\V^{\otimes k}$.
An operator $A:\V\to\V$ lifts to an operator on $\V^{\otimes k}$ as follows.  Define
\begin{equation*}
A^{\otimes k} (\vec{v}_{i_1} \otimes \vec{v}_{i_2} \otimes \cdots \otimes \vec{v}_{i_k} )
= A\vec{v}_{i_1} \otimes A\vec{v}_{i_2} \otimes \cdots \otimes A\vec{v}_{i_k}
\end{equation*}
and use the linearity of $A$ and $\otimes$ to write this in terms of
the basis vectors \eqref{eq:BasisTensors}.
An important fact is that any orthonormal basis for $\V$ yields, via \eqref{eq:BasisTensors},
an orthonormal basis for $\V^{\otimes k}$.

The $k$th \emph{symmetric tensor power} of $\V$ is the $\binom{n+k-1}{k}$-dimensional
vector space $\Sym_k \V \subset \V^{\otimes k}$ spanned by the \emph{symmetric tensors}:
\begin{equation*}
\vec{v}_{i_1} \odot \vec{v}_{i_2} \odot \cdots \odot \vec{v}_{i_k}
= \frac{1}{k!} \sum_{\sigma \in S_k} \vec{v}_{\sigma(i_1)} \otimes \vec{v}_{\sigma(i_2)} \otimes \cdots \otimes \vec{v}_{\sigma(i_k)},
\end{equation*}
where $S_k$ denotes the symmetric group on $k$ letters.  Let $A^{\Sym_k}$ denote the restriction 
$A^{\otimes k}|_{\Sym_k \V}$.

\begin{proposition}
If $d \geq 2$ is even and $A \in \H_n(\C)$, then
\begin{equation*}
\cnorm{A}_d^{d} = \tr (A^{\Sym_d}) .
\end{equation*}
\end{proposition}

\begin{proof}
Let $A:\V\to\V$ be selfadjoint with eigenvalues $\lambda_1,\lambda_2,\ldots,\lambda_n$ and corresponding
orthonormal eigenbasis $\vec{v}_1,\vec{v}_2,\ldots, \vec{v}_n$. 
Then $\vec{v}_{i_1} \odot \vec{v}_{i_2} \odot \cdots \odot \vec{v}_{i_k}$ is an eigenvector of $A^{\Sym_k}$ with eigenvalue
$\lambda_{i_1}\lambda_{i_2}\cdots \lambda_{i_k}$.  Sum over these $\binom{n+k-1}{k}$ eigenvectors and conclude that
$\tr (A^{\Sym_k}) = h_k(\lambda_1,\lambda_2,\ldots,\lambda_n)$.
\end{proof}

If $A$ is the adjacency matrix of a graph $\Gamma$, then $\cnorm{A}_d$ concerns
the $d$th symmetric tensor power of $\Gamma$, a weighted graph obtained from $\Gamma$ in a straightforward (but tedious) manner by computing the matrix representation of $A^{\Sym_d}$ with respect to the normalization of the
orthogonal basis of symmetrized tensors.

%%%%%%%%%%%%%%%%%%%%%%%%%%%%%%%%%%%%%%%%%%%%%%%%%%%%
\subsection{Equivalence constants}\label{Subsection:Constants}
Any two norms on a finite-dimensional vector space are equivalent.  Thus,
each norm $\cnorm{\cdot}_d$ on $\H_n(\C)$ (with $d\geq 2$ even) is equivalent to the operator norm $\onorm{\cdot}$.
We compute admissible equivalence constants below.

\begin{theorem}\label{Theorem:Constants}
For $A \in \H_n(\C)$ and even $d \geq 2$, 
\begin{equation*}
\bigg(\frac{1}{2^{\frac{d}{2}} (\frac{d}{2})!}\bigg)^{1/d}\onorm{A} 
\,\,\leq\,\, \cnorm{A}_d 
\,\,\leq\,\, \binom{n+d-1}{d}^{1/d}\onorm{A}
\end{equation*}
The upper inequality is sharp if and only if $A$ is a multiple of the identity.
\end{theorem}

\begin{proof}
For $A \in \H_n(\C)$ and even $d\geq 2$, the triangle inequality yields
\begin{align*}
\cnorm{A}_d^{d}
&=  h_{d}( \lambda_1(A),\lambda_2(A),\ldots,\lambda_n(A) )  \\
&= \big| h_{d}( \lambda_1(A),\lambda_2(A),\ldots,\lambda_n(A) ) \big| \\
&\leq  h_{d}( |\lambda_1(A)| , |\lambda_2(A)|,\ldots ,|\lambda_n(A)| )  \\
&\leq  h_{d}( \onorm{A}, \onorm{A},\ldots, \onorm{A} ) \\
&= \onorm{A}^{d} h_{d}(1,1,\ldots,1) \\
&=\onorm{A}^{d} \binom{n+d-1}{d} .
\end{align*}
Equality occurs if and only if $\lambda_i(A) = |\lambda_i(A)| = \onorm{A}$ for $1 \leq i \leq n$;
that is, if and only if $A$ is a multiple of the identity.

Hunter \cite{Hunter} established that 
\begin{equation*}
h_{2p}(\vec{x}) \geq \frac{1}{2^p p!} \norm{\vec{x}}^{2p},
\end{equation*}
in which $\norm{ \vec{x} }$ denotes the Euclidean norm of $\vec{x} \in \R^n$.
Let $d=2p$ and conclude 
\begin{equation*}
\cnorm{A}_d \geq  \bigg( \frac{1}{2^{\frac{d}{2}} (\frac{d}{2})!} \bigg)^{1/d} \norm{A}_{F}
\geq  \bigg( \frac{1}{2^{\frac{d}{2}} (\frac{d}{2})!} \bigg)^{1/d} \onorm{A},
\end{equation*}
in which $\norm{A}_F$ denotes the Frobenius norm of $A \in \H_n(\C)$.
\end{proof}

\begin{remark}
For $A \in \M_n(\C)$, we may apply the upper bound in Theorem \ref{Theorem:Constants} to
$e^{it} A + e^{-it} A^*$ and use \eqref{eq:gennorm} to deduce that
\begin{equation*}
\cnorm{A}_d \leq
\Bigg(\frac{ \binom{n+d-1}{d}}{2\pi {\binom{d}{d/2}}}\int_0^{2\pi} 
\onorm{ e^{it} A+e^{-it} A^*}^d\,dt\Bigg)^{1/d}
\leq 2  \Bigg(\frac{\binom{n+d-1}{d}}{ {\binom{d}{d/2}}}\Bigg)^{1/d} \onorm{A}.
\end{equation*}
\end{remark}

\begin{remark}
Hunter's lower bound was improved by Baston \cite{Baston}, who proved that
\begin{equation*}
h_{2p}(\vec{x}) \geq \frac{1}{2^p p!} \bigg( \sum_{i=1}^n x_i^2 \bigg)^p + \lambda_p \bigg( \sum_{i=1}^n x_i \bigg)^{2p}
\end{equation*}
for $\vec{x} =(x_1,x_2,\ldots,x_n) \in \R^n$,
where
\begin{equation*}
\lambda_p = \frac{1}{n^p}\left( \binom{n+2p-1}{2p} \frac{1}{n^p} - \frac{1}{2^p p!} \right) > 0.
\end{equation*}
Equality holds if and only if $p=1$ or $p\geq 2$ and all the $x_i$ are equal.
However, Baston's result does not appear to yield a significant improvement in the lower bound of Theorem \ref{Theorem:Constants}.
\end{remark}

%%%%%%%%%%%%%%%%%%%%%%
\section{Open Questions and Remarks}\label{Section:Open}

The answers to the following questions have eluded us.

\begin{problem}
What are the best constants $c_d$, independent of $n$, such that $c_d \norm{ \cdot }_d$ is submultiplicative?
Do such constants exists?  See Remark \ref{Remark:Submult}.
\end{problem}

\begin{problem}
What is the best complexified version of Theorem \ref{Theorem:Constants}?
Can the upper bound be improved (the estimate $\onorm{ e^{it} A+e^{-it} A^*} \leq 2\onorm{A}$ seems wasteful on average)?
Can we get a sharp lower bound?
\end{problem}

\begin{problem}
If one uses \eqref{eq:Extended} to evaluate $\cnorm{A}_d^d$, there are many repeated terms.
For example, $(\tr A^*A)(\tr A)(\tr A^*) = (\tr AA^*)(\tr A^*)(\tr A)$ because of the cyclic invariance of the trace
and the commutativity of multiplication.  If one chooses a single representative for each such class of expressions and simplifies,
one gets expressions such as \eqref{eq:cn2} and \eqref{eq:cn4}.
Is there a combinatorial interpretation of the resulting coefficients?
\end{problem}

For motivation, the reader is invited to consider
{\small
\begin{align*}
\cnorm{A}_6^6 &= \frac{1}{720}\Big( (\tr A)^3 \tr(A^*)^3 + 3 \tr(A) \tr(A^*)^3 \tr(A^2) \\ 
&\qquad+  9 (\tr A)^2 \tr(A^*)^2 \tr(A^*A) + 9 \tr(A^*)^2 \tr(A^2) \tr(A^*A) \\
&\qquad +  18 \tr(A) \tr(A^*) \tr(A^*A)^2 + 6 \tr(A^*A)^3 + 3 (\tr A)^3 \tr(A^*) \tr(A^{*2}) \\ &\qquad+ 
 9 \tr(A) \tr(A^*) \tr(A^2) \tr(A^{*2}) + 9 (\tr A)^2 \tr(A^*A) \tr(A^{*2}) \\ &\qquad+ 
 9 \tr(A^2) \tr(A^*A) \tr(A^{*2}) + 2 \tr(A^*)^3 \tr(A^3) \\ &\qquad+ 
 6 \tr(A^*) \tr(A^{*2}) \tr(A^3) + 18 \tr(A) \tr(A^*)^2 \tr(A^*A^2) \\ &\qquad+ 
 36 \tr(A^*) \tr(A^*A) \tr(A^*A^2) + 18 \tr(A) \tr(A^{*2}) \tr(A^*A^2)\\ &\qquad + 
 18 (\tr A)^2 \tr(A^*) \tr(A^{*2}A) + 18 \tr(A^*) \tr(A^2) \tr(A^{*2}A) \\ &\qquad+ 
 36 \tr(A) \tr(A^*A) \tr(A^{*2}A) + 36 \tr(A^*A^2) \tr(A^{*2}A) \\ &\qquad+ 
 2 (\tr A)^3 \tr(A^{*3}) + 6 \tr(A) \tr(A^2) \tr(A^{*3}) \\ &\qquad+ 
 4 \tr(A^3) \tr(A^{*3}) + 18 \tr(A^*)^2 \tr(A^*A^3)\\ &\qquad + 
 18 \tr(A^{*2}) \tr(A^*A^3) + 18 \tr(A) \tr(A^*) \tr(A^*AA^*A) \\ &\qquad+ 
 18 \tr(A^*A) \tr(A^*AA^*A) + 36 \tr(A) \tr(A^*) \tr(A^{*2}A^2)\\ &\qquad + 
 36 \tr(A^*A) \tr(A^{*2}A^2) + 18 (\tr A)^2 \tr(A^{*3}A) \\ &\qquad+ 
 18 \tr(A^2) \tr(A^{*3}A) + 36 \tr(A^*) \tr(A^*AA^*A^2) \\ &\qquad+ 
 36 \tr(A^*) \tr(A^{*2}A^3) + 36 \tr(A) \tr(A^{*2}AA^*A) \\ &\qquad+ 
 36 \tr(A) \tr(A^{*3}A^2) + 12 \tr(A^*AA^*AA^*A) + 36 \tr(A^{*2}A^2A^*A) \\ &\qquad+ 
 36 \tr(A^{*2}AA^*A^2) + 36 \tr(A^{*3}A^3) \Big). 
\end{align*}
}

\begin{remark}
The recent paper of Issa, Mourad, and Abbas \cite{MR4320503} contains
results similar to ours, but obtained with different techniques.  However, their paper deals with symmetric
gauge functions and hence invokes positivity assumptions that we have eschewed.
Remarkably, these papers were written independently and nearly simultaneously:
our paper appeared on the \texttt{arXiv} on 3 June 2021, whereas the preprint of \cite{MR4320503}
appeared on 7 June 2021.
\end{remark}

%%%%%%%%%%%%%%%%%%%%%%%%%%%%%%%%%%%%%%
%%%%%%%%%%%%%%%%%%%%%%%%%%%%%%%%%%%%%%
\bibliography{CHS-Norm}
\bibliographystyle{amsplain}

\end{document}